\documentclass[preprint,1p]{elsarticle}

\makeatletter
 \def\ps@pprintTitle{%
 	\let\@oddhead\@empty
 	\let\@evenhead\@empty
 	\def\@oddfoot{\footnotesize\itshape
 		{} \hfill\today}%
 	\let\@evenfoot\@oddfoot
 }
\makeatother
\usepackage[unicode]{hyperref}
\usepackage{latexsym}
\usepackage{indentfirst}
\usepackage{amsxtra}
\usepackage{amssymb}
\usepackage{amsthm}
\usepackage{amsmath}
\usepackage{mathrsfs} 
\usepackage{xcolor}
\usepackage{color}
\usepackage{amsfonts}
\usepackage{mathtools}

\usepackage[capitalise]{cleveref}

\newtheorem{theor}{Theorem}[section]
\newtheorem{prop}[theor]{Proposition}
\newtheorem{cor}[theor]{Corollary}
\newtheorem{lemma}[theor]{Lemma}
\theoremstyle{definition} 
\newtheorem{defin}[theor]{Definition}
\newtheorem{rem}[theor]{Remark}

\newtheorem{ex}{Example}
\newtheorem{conv}{Convention}

\DeclareMathOperator{\Sym}{Sym}
\DeclareMathOperator{\id}{id}
\DeclareMathOperator{\Aut}{Aut}
\DeclareMathOperator{\Ret}{Ret}

\begin{document}

\begin{frontmatter}
	\title{Classification of uniconnected involutive solutions of the Yang-Baxter equation with odd size and a Z-group permutation group
%\tnoteref{mytitlenote}
}
	%\title{Some indecomposable involutive solutions of the Yang-Baxter equation of multipermutation level $2$ with non-abelian regular permutation group
%\tnoteref{mytitlenote}
%}
	\author{M.~CASTELLI	\tnoteref{mytitlenote}}
	\tnotetext[mytitlenote]{The author is member of GNSAGA (INdAM).}
	\ead{marco.castelli@unisalento.it - marcolmc88@gmail.com}
	%\cortext[c1]{Corresponding author}
%		\author[unile]{}
%	\ead{}
%	\author[unile]{}
%	\ead{}
	\address{ Lecce (Italy)}

\begin{abstract}
In the first part of this paper, we investigate the retraction of finite uniconnected involutive non-degenerate set-theoretic solutions of the Yang-Baxter equation by means of left braces, giving a precise description in some cases. In the core of the paper, we also use left braces to classify all the uniconnected involutive non-degenerate set-theoretic solutions having odd size and a Z-group permutation group. As an application, we classify all the uniconnected involutive non-degenerate solutions having odd square-free size.

%Using the method to construct all the involutive solutions developed in \cite{bachiller2016solutions,rump2020}, we present a family of indecomposable involutive set-theoretic solutions of multipermutation level $2$ with non-abelian permutation group.  As an application, we classify completely all the indecomposable ones for which the permutation group has size $pq$, where $p$ and $q$ are not necessarily distinct prime numbers. Using the same approach, we also  classify the indecomposable ones having size $2p^2$ and regular permutation group.
\end{abstract}

\begin{keyword}
\texttt{set-theoretic solution\sep Yang-Baxter equation\sep brace \sep cycle set}
\MSC[2020] 
16T25\sep 81R50 \sep 20N02 \sep 20E22
\end{keyword}
% ----------------------
\end{frontmatter}

% ----------------------
\section*{Introduction}
A \emph{set-theoretic solution of the Yang-Baxter equation} on a non-empty set $X$ is a pair $\left(X,r\right)$, where 
$r:X\times X\to X\times X$ is a map such that the relation
\begin{align*}
\left(r\times\id_X\right)
\left(\id_X\times r\right)
\left(r\times\id_X\right)
= 
\left(\id_X\times r\right)
\left(r\times\id_X\right)
\left(\id_X\times r\right)
\end{align*}
is satisfied.  
The paper by Drinfel'd \cite{drinfeld1992some} moved the interest of several researchers for finding and classifying solutions of this equation in the last thirty years.
Writing a solution $(X,r)$ as $r\left(x,y\right) = \left(\lambda_x\left(y\right)\rho_y\left(x\right)\right)$, with
$\lambda_x, \rho_x$ maps from $X$ into itself, for every $x\in X$, we say that $(X, r)$ is \emph{non-degenerate} if $\lambda_x,\rho_x\in \Sym_X$, for every $x\in X$, and \emph{involutive} if $r^2=\id_{X\times X}$. Over the years, the involutive non-degenerate solutions have been widely studied starting from the seminal papers by Gateva-Ivanova and Van den Bergh \cite{gateva1998semigroups}, Gateva-Ivanova and Majid \cite{gateva2008matched}, and Etingov, Schedler, and Soloviev \cite{etingof1998set}.\\
In particular, in \cite[Section 2]{etingof1998set} the class of \textit{indecomposable solutions} was introduced. Technically, an involutive non-degenerate solution $(X,r)$ is said to be \emph{decomposable} if there exists a non-trivial partition $\{Y,Z\}$ such that $(Y,r_{|_{Y\times Y}})$ and $(Z,r_{|_{Z\times Z}})$ are again solutions; otherwise, $(X,r)$ is called \emph{indecomposable}. In the same paper, the authors studied these solutions by means of some associated algebraic structures, such as the \textit{structure group} and the \textit{associated permutation group} (see \cite{etingof1998set} for more details). In this context, a remarkable result gives a complete classification of the involutive non-degenerate solutions having a prime number of elements (see \cite[Theorem 2.13]{etingof1998set}). The paper of Etingof, Schedler and Soloviev stimulated the study of this class of solutions and in the last years several results were obtained in the classification of indecomposable solutions. In \cite{cacsp2018} and \cite{smock} a first attempt to construct examples of indecomposable involutive solutions was given by using \textit{cycle sets} and \textit{left braces}, an algebraic structures strictly related to involutive solutions (see \cite{rump2005decomposition,rump2007braces} for more details). In \cite{capiru2020} some results involving the concrete classifications of indecomposable solutions were given: indeed, the authors completely classified the indecomposable involutive solutions having size $pq$ (where $p$ and $q$ are not necessarily distinct prime numbers) and abelian permutation group. This result was extended some months later by Jedlička, Pilitowska and Zamojska-Dzenio in \cite{JePiZa20x}, where they classified all the indecomposable involutive solutions having multipermutation level $2$ and abelian permutation group. In the last year, the classification obtained in \cite{capiru2020} was extended by Rump in \cite{Ru20} and by Jedlička, Pilitowska and Zamojska-Dzenio in \cite{jedlivcka2021cocyclic} in an other direction: indeed, a complete classification of the indecomposable solutions with cyclic permutation group (without restriction on the cardinality and the multipermutation level) was exhibited. In particular, in \cite[Theorem 3.9 and Corollary 3.10]{jedlivcka2021cocyclic}, the authors provided the precise structure of the retraction of such a solutions and give a formula to compute their multipermutation level. At the same time, using different tools, in \cite{cedo2020primitive} Ced\'o, Jespers and Okni\'nski classified the indecomposable involutive solutions having primitive permutation group, showing that they coincide with the ones having a prime number of elements.\\
In this paper, we study the indecomposable involutive solutions with regular permutation group. This is also motivated by a recent paper due by Rump, in which he developed a machinery (see \cite[Theorems 2-3]{rump2020}) to classify all the indecomposable involutive solutions for which the classification of the ones with regular permutation group is the first fundamental step. Following Rump's terminology, we will call these solutions \textit{uniconnected}. In the first part, we take advantage of the results obtained in 
\cite{rump2019classification,Ru20} to study the retraction of uniconnected solutions by means of left braces, giving special attention on the solutions associated to cyclic left braces and the ones having abelian permutation group. In particular, we show that the retractions of these solutions can be completely understood in terms of the retraction of the associated left braces. As applications of these results, we answer to \cite[Question 3.11]{jedlivcka2021cocyclic} and we extend a result on retractability of solutions obtained in \cite[Proposition 2]{capiru2020}.\\
In the core of the paper, we study uniconnected involutive solutions having odd size and a Z-group permutation group by means of their associated left braces. Recall that a finite group is said to be a \emph{Z-group} if all its Sylow subgroups are cyclic. We use the classification of cyclic left braces given in \cite{rump2007classification,rump2019classification} together with \cite[Theorem 1]{rump2020} to develop a concrete construction of these solutions and we cut the isomorphism classes using a theorem due by Bachiller, Ced\'o and Jespers \cite[Theorem 4.1]{bachiller2016solutions}. In analogy to the results obtained in \cite[Theorem 3.15]{jedlivcka2021cocyclic} and \cite[Corollary 3.10]{jedlivcka2021cocyclic} for involutive  solutions with cyclic permutation group, we exhibit a set of invariants for uniconnected solutions having odd size and a Z-group permutation group and a formula to compute their multipermutation levels.
As an application, we provide a complete classification of uniconnected involutive non-degenerate solutions having square-free odd order. 

\section{Cycle sets: basic definitions and results}

In this section we give the background informations involving cycle sets that will be useful in the following. In \cite{rump2005decomposition}, a one-to-one correspondence between solutions and a special class of cycle sets, i.e., \emph{non-degenerate} cycle sets, was established. To illustrate this correspondence, let us firstly recall the following definition.

\begin{defin}[p. 45, \cite{rump2005decomposition}]
A pair $(X,\cdot)$ is said to be a cycle set if each left multiplication $\sigma_x:X\longrightarrow X,$ $y\mapsto x\cdot y$ is invertible and 
$$(x\cdot y)\cdot (x\cdot z)=(y\cdot x)\cdot (y\cdot z), $$
for all $x,y,z\in X$.  Moreover, a cycle set $(X,\cdot)$ is called \textit{non-degenerate} if the squaring map $x\mapsto x\cdot x$ is bijective.
\end{defin}

\noindent From now on, for a cycle set we mean a non-degenerate cycle set.

\begin{prop}[Propositions 1-2, \cite{rump2005decomposition}]\label{corrisp}
Let $(X,\cdot)$ be a cycle set. Then the pair $(X,r)$, where $r(x,y):=(\sigma_x^{-1}(y),\sigma_x^{-1}(y)\cdot x)$, for all $x,y\in X$, is a solution of the Yang-Baxter equation which we call the associated solution to $(X,\cdot)$. Moreover, this correspondence is one-to-one.
\end{prop}

In \cite{etingof1998set}, Etingof, Schedler and Soloviev introduced the so-called \textit{retract relation}, an equivalence relation on $X$ which we denote by $\sim_r$. If $(X,r)$ is a solution, then $x\sim_r y$ if and only if $\lambda_x=\lambda_y$, for all $x,y\in X$. In this way, we can define canonically another solution, having the quotient $X/\sim_r$ as underlying set, which is named \textit{retraction} of $(X,r)$ and is indicated by $\Ret(X, r)$. 
As one can expect, the retraction of a solution corresponds to the retraction of a cycle set. Tecnically, in \cite{rump2005decomposition} Rump defined the binary relation $\sim_\sigma$ on $X$ by $x\sim_\sigma y :\Longleftrightarrow \sigma_x = \sigma_y$ for all $x,y\in X$, and he showed that it is a congruence of $(X,\cdot)$. Moreover, he proved that the quotient $X/\sim$, which we denote by $\Ret(X)$, is a cycle set whenever $X$ is non-degenerate and he called it the \emph{retraction} of $(X,\cdot)$. As the name suggests, if $(X, \cdot)$ is the cycle set associated to a solution $(X,r)$, then the retraction $\Ret(X)$ is the cycle set associated to $\Ret(X,r)$. Besides, a cycle set $X$ is said to be \textit{irretractable} if $\Ret(X)=X$, otherwise it is called \textit{retractable}.

\medskip 

\noindent To study cycle sets (and hence involutive solutions), one can consider the retraction-process. Therefore the following definition is of crucial importance.

\begin{defin}
    A cycle set $X$ has \textit{multipermutation level $n$} if $n$ is the minimal non-negative integer such that $\Ret^n(X)$ has cardinality one, where 
    $$\Ret^0(X):=X \text{ \ and \ } \Ret^i(X):=\Ret(\Ret^{i-1}(X)) \text{, \ for }i>0. $$ 
From now on, we will write $mpl(X)=n$ to indicate that $X$ has multipermutation level $n$.
\end{defin}
\medskip

 A standard subgroup associated to a cycle set $X$ is the permutation group generated by the set $\{\sigma_x \, | \, x\in X\}$. It will be denoted by $\mathcal{G}(X)$ and we call it the \textit{associated permutation group}. Obviously, in terms of solutions, the associated permutation group is exactly the permutation group generated by the set $\{\lambda_x \, | \, x\in X\}$.
\medskip

\noindent In this paper, we focus our attention on cycle sets that are indecomposable. In particular, among these cycle sets, we will consider the uniconnected ones. 

\begin{defin}
A cycle set $(X,\cdot)$ is said to be \textit{indecomposable} if the permutation group $\mathcal{G}(X)$ acts transitively on $X$, otherwise it is said to be \textit{decomposable}. Moreover, $X$ is said to be \textit{uniconnected} if $\mathcal{G}(X)$ acts regularly on $X$.
\end{defin}

\noindent Clearly, an uniconnected cycle set is indecomposable, but the converse is not necessarily true.

\begin{rem}
In the rest of the paper, we will study indecomposable solutions by their associated cycle sets. In every case, all the results involving cycle sets can be translated in terms of solutions by \cref{corrisp}.
\end{rem}

\section{Left braces: basic definitions and results}

Now, we give the backgroud informations of left braces. We start giving the definition of these algebraic structures using the formulation contained in \cite{cedo2014braces}. For the original formulation, we refer the reader to \cite{rump2007braces}.

\begin{defin}[\cite{cedo2014braces}, Definition 1]
A set $A$ endowed of two operations $+$ and $\circ$ is said to be a \textit{left brace} if $(A,+)$ is an abelian group, $(A,\circ)$ a group, and the equality $a\circ (b + c)= a\circ b - a + a\circ c$ follows, for all $a,b,c\in A$. 
\end{defin}

\noindent 
From now on, if $(A,+,\circ)$ is a left brace, the group $(A,+)$ will be called the \textit{additive group} and the group $(A,\circ)$ will be called the \textit{multiplicative group}. An example of left brace can be constructed taking and abelian group $(A,+)$ and setting $a\circ b:=a+b$ for all $a,b\in A$. Through the paper, we call these left braces \emph{trivial}. Left braces $(A,+,\circ)$ with cyclic multiplicative group are intensively studied (see for example \cite{rump2007classification, rump2019classification}).

\begin{ex}\label{esecic}
Let $p$ be a prime number, $k,t$ natural numbers such that $t\leq k$, $(A,+)=(\mathbb{Z}/p^k\mathbb{Z},+)$ and the binary operation $\circ$ on $A$ by $a\circ b:=a+b+a b p^t$ for all $a,b\in A$. Then the triple $(A,+,\circ)$ is a left brace and $(A,\circ)$ is a cyclic group.  
\end{ex}

\noindent In \cite{rump2007classification}, Rump showed that every left brace with prime-power size, for some odd prime number, and cyclic multiplicative group can be constructed as in the previous example: for this reason, we will indicate this left brace as $B(p,k,t)$. Moreover, it is well-known that every left brace with odd size and cyclic multiplicative group is isomorphic to the direct product of left braces having prime-power size and cyclic multiplicative group (for more details, see \cite{rump2007classification,rump2019classification}).\\
A simple way to construct other left braces, which will be useful in the next sections, is the semidirect product of left braces. Given a left brace $(A,+,\circ)$, an automorphism of the left brace $A$ is a bijective function $\alpha$ from $A$ to itself such that $\alpha\in Aut(A,+)\cap Aut(A,\circ)$. Of course, the set consisting of all the automorphisms of $(A,+,\circ)$ is a subgroup of the permutations group on $A$ and we indicate it by $Aut(A,+,\circ)$. If $A_1$ and $A_2$ are left braces and $\alpha$ a homomorphism from $(A_2,\circ)$ to $Aut(A_1,+,\circ)$, the \emph{semidirect product of the left braces $A_1$ and $A_2$ via $\alpha$} is the left brace $(A_1\times A_2,+,\circ)$ such that $(A_1\times A_2,+)$ is the direct product of the additive groups $(A_1,+)$ and $(A_2,+)$ and $(A_1\times A_2,\circ)$ is the semidirect product of the multiplicative groups $(A_1,\circ)$ and $(A_2,\circ)$ via $\alpha$. In this way, the operations $+$ and $\circ$ of the semidirect product are given by
$$(a_1,a_2)+(b_1,b_2):=(a_1+b_1,a_2+b_2) $$
$$(a_1,a_2)\circ(b_1,b_2):=(a_1\circ b_1^{a_2},a_2\circ b_2) $$
for all $(a_1,a_2),(b_1,b_2)\in A_1\times A_2 $, where $b_1^{a_2} $ simply means $\alpha(a_2)(b_1)$, as the function $\alpha$ is clear from the context.

\medskip

\noindent Given a left brace $A$ and $a\in A$, let us denote by $\lambda_a:A\longrightarrow A$ the map from $A$ into itself defined by $\lambda_a(b):= - a + a\circ b,$ for all $b\in A$. As shown in \cite[Proposition 2]{rump2007braces} and \cite[Lemma 1]{cedo2014braces}, these maps have special properties. We recall them in the following proposition.
\begin{prop}\label{action}
Let $A$ be a left brace. Then, the following are satisfied: 
\begin{itemize}
\item[1)] $\lambda_a\in\Aut(A,+)$, for every $a\in A$;
\item[2)] the map $\lambda:A\longrightarrow \Aut(A,+)$, $a\mapsto \lambda_a$ is a group homomorphism from $(A,\circ)$ into $\Aut(A,+)$.
\end{itemize}
\end{prop}

\noindent The map $\lambda$ is of crucial importance to construct cycle sets (and hence involutive solutions) using left braces, as one can see in the following propositions.

\begin{prop}[Lemma 2, \cite{cedo2014braces}]\label{costru1}
Let A be a left brace and $\cdotp$ the binary operation on $A$ given by 
$$a\cdotp b:=\lambda_a^{-1}(b)$$
for all $a,b\in A$. Then, $(A,\cdotp)$ is a decomposable cycle set.
\end{prop}

\noindent Before giving the following result, recall that a subset $X$ of a left brace $(A,+,\circ)$ is said to be a \emph{cycle base} if $X$ is a union of orbits respect to the action $\lambda$ and generates the additive group $(A,+)$. Moreover, a cycle base that is a single orbit is called \emph{transitive cycle base}.

\begin{prop}[Theorem 2, \cite{rump2020}]\label{rump20}
Let A be a left brace, $X$ a transitive cycle base and $g\in X$. Define on $A$ the binary operation $\bullet$ 
\begin{equation}\label{forind}
    a\bullet b:=(\lambda_a(g))^{-}\circ b,
\end{equation}
for all $a,b\in A$. Then, $(A,\bullet)$ is an uniconnected cycle set such that $\mathcal{G}(A)\cong (A,\circ)$. Conversely, every uniconnected cycle set can be constructed in this way.
\end{prop}

\noindent In the two previous propositions, we saw that a left brace $A$ provide a decomposable cycle set and several indecomposable cycle sets, which in full generality are not isomorphic. The cycle set provided by \cref{costru1} will be called the \emph{decomposable associated cycle set}, and similarly a cycle set provided by \cref{rump20} will be called an \emph{uniconnected associated cycle set}.

\bigskip

Similarly to rings, the notion of ideal of a left brace was introduced in \cite{rump2007braces} and reformulated in \cite[Definition 3]{cedo2014braces}.
\begin{defin}
Let $A$ be a left brace. A subset $I$ of $A$ is said to be a \textit{left ideal} if it is a subgroup of the multiplicative group and $\lambda_a(I)\subseteq I$, for every $a\in A$. Moreover, a left ideal is an \textit{ideal} if it is a normal subgroup of the multiplicative group.
\end{defin}

\noindent Given an ideal $I$, it holds that the structure $A/I$ is a left brace called the \emph{quotient left brace} of $A$ modulo $I$. A special ideal of a left brace, introduced in \cite{rump2007braces}, is the socle. In the terms of \cite[Section 4]{cedo2014braces}, it is the following.

\begin{defin}
Let $A$ be a left brace. Then, the set 
$$
    Soc(A) := \{a\in A \ | \ \forall \,     b\in A \quad  a + b = a\circ b \}
$$
is named \emph{socle} of $A$.
\end{defin}

\noindent Clearly, $Soc(A) := \{a\in A \ | \ \lambda_a = \id_A\}$. It is well-known that $Soc(A)$ is an ideal of $A$ and it is strongly related to the retraction of its decomposable associated cycle set. The following proposition is implicitly contained in \cite[Proposition 7]{rump2007braces} and \cite[Lemma 3]{cedo2014braces}.

\begin{prop}\label{propzoc}
Let $A$ be a left brace and $(A,\cdotp)$ the decomposable associated cycle set. If $(A/ Soc(A),\cdotp)$ is the decomposable associated cycle set of the left brace $A/Soc(A)$, then $(A/Soc(A), \cdotp)$ coincides with the retraction $\Ret(A,\cdotp)$ of $(A,\cdotp)$.
\end{prop}

\begin{defin}
A left brace $A$ has \emph{finite multipermutation level} if the decomposable associated cycle set $(A,\cdotp)$ has finite multipermutation level. If $n$ is a natural number and a left brace $A$ is such that the decomposable associated cycle set $(A,\cdotp)$ has multipermutation level equal to $n$, then $A$ is said to be of multipermutation level $n$.
\end{defin}

\begin{ex}\label{esecicmpl}
Let $A$ be the left brace $B(p,k,t)$ given in \cref{esecic}. Then, the multipermutation level of $A$ is $ \lceil{\frac{k}{t}}\rceil $, where, if $z$ is a real number, $ \lceil{z}\rceil  $ is the greatest integer part of $z$ (see \cite[Corollary 3.10]{jedlivcka2021cocyclic}). Therefore, if $C$ is a left brace isomorphic to the direct product $B(p_1,k_1,t_1)\times...\times B(p_n,k_n,t_n) $, for some distinct prime numbers $p_1,...,p_n$, then the multipermutation level is equal to $max\{\lceil{\frac{k_1}{t_1}}\rceil,...,\lceil{\frac{k_n}{t_n}}\rceil \} $.
\end{ex}

\medskip

We conclude the section focusing our attention on \cref{rump20}. Note that the cycle sets obtained using this results can be also obtained using a construction developed in \cite{bachiller2016solutions}. Indeed, if in Theorem 3.1 of \cite{bachiller2016solutions} we take $I:=\{i\}$ as a set of size one given by a transitive cycle base $i$, $a_i$ an element of $i$, $J_i:=\{j\}$ a set of size one and $K_{i,j}:=\{0\}$ the trivial subgroup of the stabilizer of $a_i$ under the action $\lambda$, we can construct all the cycle sets provided in \cref{rump20}. Moreover, in \cite[Theorem 4.1]{bachiller2016solutions} a method to verify when two cycle sets, having permutation groups isomorphic as left braces, are isomorphic is provided. In the following we reformulate the result in our case.

\begin{prop}[Theorem 4.1, \cite{bachiller2016solutions}]\label{isom}
Let A be a left brace, $X_1$ and $X_2$ a transitive cycle bases, $g_1\in X_1$ and $g_2\in X_2$. Let $A_1$ (resp. $A_2$) be the cycle set constructed as in \cref{rump20} starting from $g_1$ and $X_1$ (resp. $g_2$ and $X_2$).  Then, $A_1$ and $A_2$ are isomorphic if and only if there exist $z\in A$ and $\psi\in Aut(A,+,\circ)$ such that $\psi(x_1)=\lambda_z(x_2)$.
\end{prop}

%\begin{cor}\label{isom2}
%Let A be a left brace, $X_1$ and $X_2$ a transitive cycle bases, $g_1\in X_1$ and $g_2\in X_2$. Let $A_1$ (resp. $A_2$) be the cycle set constructed as in \cref{rump20} starting from $g_1$ and $X_1$ (resp. $g_2$ and $X_2$).  If $g_1$ and $g_2$ are in the  same transitive cycle base or if $g_1$ and $g_2$ are in the same orbit respect to the action of $Aut(A,+,\circ)$, then $A_1$ is isomorphic to $A_2$.
%\end{cor}

\section{Retraction of uniconnected cycle sets}

In this section, we study the retraction of uniconnected cycle sets by means of left braces. We show that, in some cases, the underlying set of the retraction of these cycle sets is the quotient of the associated left brace by its socle and we use these facts to give an answer to \cite[Question 3.11]{jedlivcka2021cocyclic}. The control of the retraction-structure of these cycle sets will be useful in the last part of the section, in which we extend the results obtained in \cite{capiru2020} on the retractability of indecomposable cycle sets with abelian permutation group.

\begin{lemma}\label{risu2}
Let $A$ be a left brace. Suppose that $A$ has a transitive cycle base and let $(A,\bullet)$ be an uniconnected associated cycle set. Moreover, let $g\in A$ be an element of a transitive cycle base such that $a\bullet b=\lambda_a(g)^{-}\circ b$ for all $a,b\in A$. Then, $\Ret(A,\bullet)$ coincides with the set of the left cosets $A/H$, where $H$ is equal to the subgroup of $(A,\circ)$ given by $H:=\{h\ |\ h\in A,\ \lambda_h(g)=g \}$. Moreover, the socle of $A$ is contained in $H$.
\end{lemma}

\begin{proof}
By a standard calculation, one can show that $H$ is a subgroup of $(A,\circ)$ containing $Soc(A)$. Since $\lambda$ is a homomorphism from $(A,\circ)$ to $Aut(A,+)$, it follows that $\sigma_a=\sigma_b$ if and only if $\lambda_{a^{-}\circ b}(g)=\lambda_{a^{-}}(\lambda_b(g))=g$, therefore the retract relation coincides with the equivalence relation induced by $H$.
\end{proof}

\noindent In general, we are not able to state if for a left brace $A$ the subgroup $H$ of the previous proposition is equal to $Soc(A)$. However, this surely happens in some cases, which we see in the following results. At first, we need a lemma. Recall that if $A$ is a \emph{cyclic left brace}, i.e. a left brace with cyclic additive group, then it surely has a transitive cycle base: it is sufficient to consider the orbit of a generator $a$ of the additive group respect to the action $\lambda$.

\begin{lemma}\label{gentrans}
Let $(A,+,\circ)$ be a finite cyclic left brace, $X$ a transitive cycle base and $g\in X$. Then, $g$ is a generator of the additive group $(A,+)$. 
\end{lemma}

\begin{proof}
If $g$ has order $n<|A|$, since the maps $\lambda_x$ preserve the order in the additive group, it follows that all the elements of $X$ have additive order equal to $n$ and therefore all the elements of $A$ have additive order that divides $n$, a contradiction.
\end{proof}

\begin{prop}\label{risu3}
Let $A$ be a finite cyclic left brace and $(A,\bullet)$ an uniconnected associated cycle set. Moreover, let $g\in A$ be an element of a transitive cycle base such that $a\bullet b=\lambda_a(g)^{-}\circ b$ for all $a,b\in A$ and $H$ the subgroup of $(A,\circ)$ such that $\Ret(A,\bullet)=A/H $. Then $H$ is equal to $Soc(A)$.
\end{prop}

\begin{proof}
The inclusion $Soc(A)\subseteq H$ follows by \cref{risu2}. By \cref{gentrans}, if $h$ is an element of $H$ the equality $\lambda_h(g)=g $ implies $\lambda_h=id_A$, therefore $H\subseteq Soc(A)$.
\end{proof}

\begin{prop}\label{risu4}
Let $A$ be a finite left brace. Suppose that $A$ has a transitive cycle base and let  $(A,\bullet)$ be an uniconnected associated cycle set. Moreover, let $g\in A$ be an element of a transitive cycle base $X$ such that $a\bullet b=\lambda_a(g)^{-}\circ b$ for all $a,b\in A$ and let $H$ be the subgroup of $(A,\circ)$ such that $\Ret(A,\bullet)=A/H $. If $H $ is a normal subgroup of $(A,\circ)$, the subgroup $H$ is equal to $Soc(A)$.
\end{prop}

\begin{proof}
The inclusion $Soc(A)\subseteq H$ follows by \cref{risu2}. Now, all the elements of $X$ are in the same orbit under the action $\lambda$, hence their stabilizers are conjugated in $(A,\circ)$. Since the stabilizer of $g$, which is equal to $H$, is a normal subgroup of $(A,\circ)$, it follows that $H$ fixes all the elements of $X$. The thesis follows since the additive group generated by $X$ is $A$.
\end{proof}

\noindent The previous propositions are useful to understand, in some cases, the relation between the retraction-process of the decomposable associated cycle set $(A,\cdotp)$ and an uniconnected associated cycle set $(A,\bullet)$. 

\begin{theor}\label{retrcyc}
Let $A$ be a finite cyclic left brace, $(A,\cdotp)$ the decomposable associated cycle set and $(A,\bullet)$ an uniconnected associated cycle set. Moreover, let $g\in A$ an element of a transitive cycle base such that $a\bullet b=\lambda_a(g)^{-}\circ b$ for all $a,b\in A$. Then, $(A,\cdotp)$ and $(A,\bullet)$  are multipermutation cycle sets and $mpl(A,\cdotp)=mpl(A,\bullet)$. Moreover, the underlying sets of $\Ret^n(A,\cdotp)$ and $\Ret^n(A,\bullet)$ coincide, for every $n\in \mathbb{N}$.
\end{theor}

\begin{proof}
Since $(A,\bullet)$ has regular permutation group, by \cite[Corollary 3.3]{RaVe21} it is retractable, hence the subgroup $H$ of $(A,\circ)$ such that $\Ret(A,\bullet)=A/H$ is not trivial. Moreover, by \cref{risu3} and \cref{propzoc}, $H$ is equal to $Soc(A)$ and $\Ret(A,\cdotp)$ and $\Ret(A,\bullet)$ have as underlying set the quotient $A/Soc(A)$. By standard calculation, if we indicate by $(A/Soc(A),\bullet)$ the uniconnected cycle set associated to the left brace $A/Soc(A)$ given by 
$$(a+Soc(A))\bullet (b+Soc(A)):=\lambda_{a+Soc(A)}(g+Soc(A))^{-}\circ (b+Soc(A)),$$
it follows that $\Ret(A,\bullet)=(A/Soc(A),\bullet)$. Moreover, the quotient left brace $A/Soc(A)$ is again a cyclic left brace. Therefore, the thesis follows by an easy induction on $|A|$.
\end{proof}

\noindent In \cite{rump2019classification} Rump used the classification of cyclic left braces to show that they have not trivial socle. Here, we use the previous theorem to show this result in a different way. 

\begin{cor}[pg. 319, \cite{rump2019classification}]
Let $A$ be a finite cyclic left brace. Then, the socle of $A$ is not trivial.
\end{cor}

\begin{proof}
By the previos theorem, the decomposable associated cycle set is retractable hence the thesis follows by \cref{propzoc}.
\end{proof}

In the following, we consider left braces $A$ whose multiplicative group is a so-called \emph{Dedekind group}, i.e., a group in which every subgroup is normal (see \cite[pg. 143]{robinson2012course}). 

\begin{theor}\label{teocons}
Let $A$ be a finite left brace such that $(A,\circ)$ is a Dedekind group and $(A,\cdotp)$ the decomposable associated cycle set. Suppose that $A$ has a transitive cycle base and let $(A,\bullet)$ be an uniconnected associated cycle set. Moreover, let $g\in A$ be an element of a transitive cycle base such that $a\bullet b=\lambda_a(g)^{-}\circ b$ for all $a,b\in A$. Then, $(A,\bullet)$ and $(A,\cdotp)$ are multipermutation cycle sets and $mpl(A,\cdotp)=mpl(A,\bullet)$. Moreover, the underlying sets of $\Ret^n(A,\cdotp)$ and $\Ret^n(A,\bullet)$ coincide, for every $n\in \mathbb{N}$.
\end{theor}

\begin{proof}
Since $(A,\bullet)$ has regular permutation group, by \cite[Corollary 3.3]{RaVe21} it is retractable, hence the subgroup $H$ of $(A,\circ)$ such that $\Ret(A,\bullet)=A/H$ is not trivial. By \cref{risu4} and \cref{propzoc}, $H$ is equal to $Soc(A)$ and $\Ret(A,\cdotp)$ and $\Ret(A,\bullet)$ have as underlying set the quotient $A/Soc(A)$. Similarly to the previous corollary,  if we indicate by $(A/Soc(A),\bullet)$ the uniconnected cycle set associated to the left brace $A/Soc(A)$ given by 
$$(a+Soc(A))\bullet (b+Soc(A)):=\lambda_{a+Soc(A)}(g+Soc(A))^{-}\circ (b+Soc(A)),$$
it follows that $\Ret(A,\bullet)=(A/Soc(A),\bullet)$. Moreover, $\{\lambda_x(g)+Soc(A)\ | \ x\in A \}$ is a transitive cycle base of $A/Soc(A)$ and $(A/Soc(A),\circ)$ is again a Dedekind group. Therefore, the thesis follows by an easy induction on $|A|$.
\end{proof}

\begin{rem}
\begin{itemize}
    \item[1)] In \cite[Question 3.11]{jedlivcka2021cocyclic}, the authors noted that, given a cyclic left brace $A$ having prime-power size and abelian multiplicative group, the multipermutation level of the decomposable associated cycle set coincides with the one of an uniconnected associated cycle set and they asked if this is a coincidence or there is a more general mechanism hidden behind. Theorems \ref{retrcyc} and \ref{teocons} show that this is not a mere coincidence.
    \item[2)]  Theorems \ref{retrcyc} and \ref{teocons} also shows that the retraction of the cycle sets belonging to these classes can be simply obtained considering the quotient module the socle by the associated left brace. By an iteration-process, this idea allow to obtain the $m^{th}$-retraction, for all $m\in \mathbb{N}$. In this way, one can obtain \cite[Theorem 3.9]{jedlivcka2021cocyclic} as a corollary.
\end{itemize}
\end{rem} 

We close the section extending the result on retractability obtained in \cite[Proposition 2]{capiru2020} for indecomposable cycle sets with abelian permutation group.

\begin{theor}
Let $X$ be an uniconnected cycle set and suppose that the permutation group  $\mathcal{G}(X)$ is a Dedekind group. Then, $X$ is a multipermutation cycle set.
\end{theor}

\begin{proof}
Let $A$ be the left brace constructed as in \cref{rump20} such that an indecomposable associated cycle set $(A,\bullet)$ is isomorphic to $X$. Then, the group $(A,\circ)$ is isomorphic to $\mathcal{G}(X)$, hence the thesis follows by \cref{teocons} applied to the left brace $A$.
\end{proof}

\noindent The following example, which closes the section, ensures the existence of an uniconnected cycle set having permutation group isomorphic to a non-abelian Dedekind group.

\begin{ex}
Let $(A,+,\circ)$ the left brace of size $8$ given by $(A,+):=(\mathbb{Z}/8\mathbb{Z},+)$ and $(A,\circ)$ the group structure on $\mathbb{Z}/8\mathbb{Z}$ given by 
$$a_1\circ a_2:=a_1 + 3^{a_1}a_2 $$
for all $a_1,a_2\in A$. Then, by \cite[Theorem 3.1]{bachiller2015classification} $(A,+,\circ)$ is a left brace such that $(A,\circ)$ is isomorphic to the quaternion group of size $8$. Since $(A,+)$ is cyclic, the left brace has a transitive cycle base. Therefore, if $(A,\bullet)$ is an uniconnected associated cycle set, by \cref{rump20} we have that $\mathcal{G}(A)$ is isomorphic to $(A,\circ)$.
\end{ex}

\section{The cyclic left braces with odd size and a Z-group permutation group}

This small section, which does not contains new "genuine" results, aims to show that cyclic left braces with odd size and a Z-group permutation group can be constructed as a particular semidirect product of left braces. The results contained in this section are implicitly contained in \cite{rump2019classification}, but we exhibit them in a comfortable form for our scopes.

\smallskip 

\noindent We start recalling that in \cite{rump2019classification} Rump showed that a cyclic left brace $A$ is isomorphic to an iterated semidirect product of cyclic left braces $(...(B_m\rtimes B_{m-1})...)\rtimes B_2)\rtimes B_1 $, where the left braces $B_i$ are constructed on a Sylow basis of the multiplicative group $(A,\circ)$. Given the left braces $B_1,..., B_m$ having respectively size $p_1^{\alpha_1},...,p_m^{\alpha_m}$ with $p_i<p_{i+1}$, the whole cyclic left brace $A$ is completely determined by the action of $B_i$ on $B_j$ (see \cite[Sections 3-4]{rump2019classification}), which we denote by $\alpha_{i,j}$, for all $i,j\in \{1,...,m\}$ with $j>i$. By \cite[Proposition 5]{rump2019classification} we have that if $B_j$ is not trivial as left brace, then the action $\alpha_{i,j}$ must be trivial for all $i\in \{1,...,m\}$ with $j>i$. Moreover, in \cite[Proposition 14]{rump2019classification} it was shown that the socle of $(...(B_m\rtimes B_{m-1})...)\rtimes B_2)\rtimes B_1 $ is the ideal $S_m\times .... \times S_1 $ where $S_i= Soc(B_i)\cap C_{B_i}(B_{i+1}\times ... \times B_{m},\circ)$ and $C_{B_i}(B_{i+1}\times ... \times B_{m},\circ) $ is the set of the elements of $(B_i,\circ)$ which commutes with the subgroup $(B_{i+1}\times ... \times B_{m},\circ) $, for all $i\in \{1,...,m \}$. The following proposition is useful to construct a family of cyclic left braces. 
%All these facts are useful to construct a family of cyclic left braces having multipermutation level at most $2$, which we exhibit in the following proposition.

\begin{prop}\label{costrtutt}
Let $r,m,v\in \mathbb{N}\cup \{0\}$ such that $m\leq r$. Moreover, let $\alpha_1,...,\alpha_r,\gamma_1,...,\gamma_v$ be natural numbers and $P:=\{p_1,...,p_r,q_1,...,q_v \}$ a set of distinct odd prime numbers such that $p_i< p_{i+1}$, for all $i=1,...,r-1$.\\
Suppose that $A_1,..., A_v$ are cyclic left braces and that $|A_i|=q_i^{\gamma_i}$ for all $i=1,...,v$. Let $B_1,...,B_r$ be cyclic left braces of size respectively $p_1^{\alpha_1},..., p_r^{\alpha_r} $ such that $B_i$ has multipermutation level $1$ for all $i=m+1,...,r$ and suppose that $\alpha$ is a homomorphism from $(B_1\times ... \times B_m,\circ)$ to $Aut(B_{m+1}\times ... \times B_{r},+)$ such that $\alpha_{|_{B_i}}$ is not trivial for all $i\in \{1,...,m\}$ and $|\alpha(B_1\times ... \times B_m)\cap Aut(B_i)|>1$, for all $i\in \{m+1,...,r \}$.\\
Denote by $\bar{A}$ the direct product of left braces $(A_1\times ... \times A_v) $ and by $\bar{B}$ the semidirect product of groups $ (B_{m+1} \times ... \times B_v) \rtimes_{\alpha}  (B_1\times ... \times B_m) $.\\
Then, $\alpha$ makes the structure $\bar{A} \times \bar{B}$ onto a cyclic left brace. 
\end{prop}

\begin{proof}
%Let $\bar{A}$ be the direct product of left braces $A_1\times ...\times A_v$ and $\bar{B}$ the semidirect product of the groups $(B_{m+1} \times ... \times B_v,+)$ and $(B_1\times ... \times B_m,\circ)$ by $\alpha$. 
Since $B_{m+1}\times ... \times B_{v} $ is a trivial left brace, $\alpha$ is a homomorphism from  $(B_1\times ... \times B_m,\circ)$ to $Aut(B_{m+1}\times ... \times B_{r},+,\circ)$, therefore it makes $A:=\bar{A} \times \bar{B}$ onto a left brace structure with cyclic additive group. 
\end{proof}

\noindent As a weaker converse of \cref{costrtutt}, we provide the following result.

\begin{prop}\label{cosZgroup}
Let $(A,+,\circ)$ be a cyclic left brace having odd size and such that $(A,\circ)$ is a Z-group. Then, $A$ can be constructed as in \cref{costrtutt}.
\end{prop}

\begin{proof}
By \cite[Chapter 9]{hall2018theory}, $(A,\circ)$ is isomorphic to a semidirect product of two cyclic groups $C_1$ and $C_2$ having coprime order. Let $r,m,\alpha_1,...,\alpha_r\in \mathbb{N}$, $p_1,...,p_r$ distinct prime numbers such that $m\leq r$, $C_2\cong \mathbb{Z}/p_1^{\alpha_1}\mathbb{Z}\times ... \times \mathbb{Z}/p_m^{\alpha_m}\mathbb{Z}$ and $C_1\cong \mathbb{Z}/p_{m+1}^{\alpha_{m+1}}\mathbb{Z}\times ... \times \mathbb{Z}/p_r^{\alpha_r}\mathbb{Z}$. Then, $(A,\circ)$ is isomorphic to $ C_1\rtimes_{\beta} C_2$ for some homomorphism $\beta$ from $C_2$ to $Aut(C_1)$. If $B_{m+1},...,B_r$ are the Sylow subgroups of $C_1$ and $B_1,...,B_m$ are the ones of $C_2$, the set $\{B_1,...,B_r\}$ can be identified with a Sylow basis of $A$, hence by \cite[Sections 3-4]{rump2019classification} the whole left brace is determined by $\beta$ and we can suppose that $p_i< p_{j}$ for all $i\in \{1,...,m\}$ and $j\in \{m+1,...,r \}$ whenever the induced action of $B_i$ on $B_j$ is not trivial. Let $S$ be the subset of $\{B_1,...,B_r\}$ given by the elements $B_j$ such that $B_j\subseteq Ker(\beta)$ if $j\in \{1,...,m \}$ or $|\beta(A)\cap Aut(B_j)|=1$ if $j\in \{m+1,...,r \}$. Define $\bar{A}$ as the subgroup of $(A,\circ)$ given by $\bar{A}:=\Pi_{B_i\in S} B_i$ and $\bar{B}$ the subgroup of $(A,\circ)$ given by $\bar{B}:=(\Pi_{B_j\in \{B_{m+1},...,B_r \}\setminus S } B_j)\rtimes_{\alpha}(\Pi_{B_i\in \{B_1,...,B_m \}\setminus S} B_i)$, where $\alpha$ is the restriction of $\beta$ onto $\Pi_{B_i\in \{B_1,...,B_m \}\setminus S} B_i$. By \cite[Proposition 2]{rump2019classification} $B_j$ must be a trivial left brace for all $B_j\in \{B_{m+1},...,B_r \}\setminus S $. Then, as left brace, $A$ is isomorphic to 
$\bar{A}\times \bar{B}$. By a standard verification, one can check that $\bar{A}, \bar{B}$ and $\alpha$ are the same as in \cref{costrtutt}.
\end{proof}

\section{Classification of uniconnected cycle sets with odd size and a Z-group permutation group}

In this section we construct and classify, by means of left braces, the uniconnected cycle sets with odd size and a Z-group permutation group. Moreover, we also use left braces to provide a formula useful to compute the multipermutation level of these cycle sets. In that regard, we start showing that the multipermutation level is always finite.

\begin{cor}\label{multilevelcyc}
Every uniconnected cycle set having odd size and with a Z-group permutation group has finite multipermutation level, which is equal to the one of its associated left brace.
\end{cor}

\begin{proof}
Let $X$ be an uniconnected cycle set having odd size and with a Z-group permutation group. By \cite[Corollary 2]{rump2019classification} the associated left brace has cyclic additive group, hence the thesis follows by \cite[Corollary at pg. 319]{rump2019classification} and \cref{retrcyc}.
\end{proof}

\begin{conv}\label{conve}
In order to give the following results, we fix some notations. If $(A,+,\circ)$ is a cyclic left brace as in \cref{cosZgroup}, we indicate by $\bar{A}$ and $\bar{B}$ the left braces constructed as in \cref{costrtutt} such that $A=\bar{A}\times \bar{B}$. In this way, $\bar{B}$ can be written as a semidirect product of left braces $(B_{m+1}\times...\times B_r)\rtimes_{\alpha} (B_1\times...\times B_m)$, where $\alpha,B_1,...,B_r$ clearly are as in \cref{costrtutt}. Following the same proposition, $\bar{A}$ can be written as direct product of the left braces $A_1,...,A_v$. Moreover, there exist a prime numbers, which we indicate by $q_1,...,q_v,p_1,...,p_r$, and a natural numbers, which we indicate by $\gamma_1,...,\gamma_v,\beta_1,...,\beta_r$, such that $|A_i|=q_i^{\gamma_i}$ for all $i\in \{1,...,v\}$ and $|B_i|=p_i^{\beta_i}$ for all $i\in \{1,...,r\}$.\\
We indicate by $I_j$ the subgroup of the left brace $B_j$ given by $I_j:=Soc(B_j)\cap Ker(\alpha) $, for all $j\in \{1,...,m\}$. Finally, we will indicate by $d_i$ the natural number such that $|Soc(A_i)|=q^{d_i}$, for all $i\in \{1,...,v\}$, and similarly we indicate by $f_i$ the natural number such that $|Soc(B_i)|=p^{f_i}$, for all $i\in \{1,...,m\}$, and by $f_i'$ the natural number such that $|I_i|=p^{f_i'}$, for all $i\in \{1,...,m\}$.
\end{conv}

By \cref{multilevelcyc}, every uniconnected cycle set having odd size and with a Z-group permutation group has finite multipermutation level: in the following, we provide a formula to compute this number. At first, we show a corollary useful for this scope.

\begin{cor}
The socle of $A$ is equal to $Soc(\bar{A})\times (B_{m+1}\times... B_r\times I_1\times...\times I_m)$. In particular, $A/Soc(A)$ has cyclic multiplicative group.
\end{cor}

\begin{proof}
Since $A$ is equal to the direct product of left braces $\bar{A}\times \bar{B}$, the socle of $A$ is equal to the direct product $Soc(\bar{A})\times Soc(\bar{B})$. By \cite[Proposition 14]{rump2019classification}, we have that $Soc(\bar{B})$ is equal to $B_{m+1}\times... B_r\times I_1\times...\times I_m$. In this way, the quotient $(A/Soc(A),\circ)$ is the group $(\bar{A}/Soc(\bar{A})\times B_1/I_1\times... \times B_m/I_m,\circ) $, hence it is cyclic.
\end{proof}

\begin{theor}\label{multilevel}
The following equalities hold: 
$$mpl(X)=mpl(A)=max\{max_{j\in \{1,...,v \}}\{\lceil{\frac{\gamma_j-d_j}{d_j}}\rceil \},\lceil{\frac{\beta_1-f'_1}{f_1}}\rceil,...,\lceil{\frac{\beta_m-f'_m}{f_m}}\rceil \}+1,$$
where $X$ is an uniconnected associated cycle set of $A$.
\end{theor}

\begin{proof}
Clearly $mpl(A)$ is equal to $1+mpl(A/Soc(A))$. Now,  $A_1,...,A_v,B_1,...,B_m$ have distinct prime-power orders; moreover, $|A_i/Soc(A_i)|=q^{\gamma_i-d_i}$ for all $i\in \{1,...,v\}$ and $|B_i/Soc(B_i)|=p^{\beta_i-f'_i}$ for all $i\in \{1,...,m\}$. Since $A/Soc(A)$ has cyclic multiplicative group, by the previous proposition and \cref{esecicmpl} it follows that $mpl(A/Soc(A))$ is the number 
$$max\{max_{j\in \{1,...,v \}}\{\lceil{\frac{\gamma_j-d_j}{d_j}}\rceil \},\lceil{\frac{\beta_1-f'_1}{f_1}}\rceil,...,\lceil{\frac{\beta_m-b'_m}{b_m}}\rceil \},$$
hence the second equality of the statement follows. The first equality follows by \cref{retrcyc}.
\end{proof}

Clearly, if $A$ as cyclic group, i.e. $\bar{B}$ is the trivial left brace $\{0\}$, the multipermutation level of $A$ is equal to $max_{j\in \{1,...,v \}}\{\lceil{\frac{\gamma_j-d_j}{d_j}}\rceil \}+1$: this is consistent with \cite[Corollary 3.10]{jedlivcka2021cocyclic}.

\begin{ex}
Let $p_1,p_2$ be distinct odd prime numbers, with $p_1<p_2$ and suppose that $p_1$ divides $p_2-1$. Suppose that $m=1$, $r=2$, $\bar{A}=\{0\}$ and $\bar{B}$ is the non-abelian semidirect product of the two trivial left braces $B_1$ and $B_2$ having size $p_1$ and $p_2$, respectively. Then, $\beta_1=\beta_2=1$, $f_1=1$, $f_1'=0$ and $I_1=\{0\}$, therefore $A=\bar{B}$ and $mpl(A)=\lceil{\frac{\beta_1-f'_1}{f_1}}\rceil+1=2$.
\end{ex}

\noindent Now, in the rest of the section we want to classify all the uniconnected cycle set having odd size and with a Z-group permutation group, that can be obtained by using the results contained in \cite{rump2019classification, rump2020}

%Therefore, the previous results provide a method to construct all the uniconnected cycle set with a Z-group permutation group having odd size.

\begin{cor}\label{deriv}
Let $(A,+,\circ)$ be a cyclic left brace as in \cref{cosZgroup}. Then, an uniconnected associated cycle set $X$ has odd size and $\mathcal{G}(X)$ is a Z-group.\\
Conversely, every uniconnected cycle set having odd size and with a Z-group permutation group can be constructed in this way.
\end{cor}

\begin{proof}
The first part is a consequence of \cref{rump20}. Conversely, let $X$ be an uniconnected cycle set of odd size and with a $Z-$group permutation group. Let $A$ be the left brace, as in \cref{rump20}, such that $X$ is an indecomposable associated cycle set of $A$. Then, by \cref{rump20} $A$ has a Z-group multiplicative group, it has odd size and by \cite[Corollary 2]{rump2019classification} $A$ is a cyclic left brace, hence it can be constructed as in \cref{cosZgroup}. 
\end{proof}

The previous corollary states that, if we know all the cyclic left braces with Z-group multiplicative group, we can construct all the uniconnected cycle set having odd size and with a Z-group permutation group. Hence, we have to cut all the isomorphic cycle sets obtained by the same cyclic left braces. In this direction we exhibit three preliminary results.

\begin{lemma}\label{lambdamap}
Let $a_1$ be an element of $\bar{A}$ and $(a_2,a_3)$ an element of   $\bar{B}= (B_{m+1}\times...\times B_r)\rtimes_{\alpha} (B_1\times...\times B_m)$, where $a_2\in  B_{m+1}\times...\times B_r$ and $a_3\in B_1\times...\times B_m$. Then, $\lambda_{(a_1,a_2,a_3)}(b_1,b_2,b_3)$ is equal to $(\lambda_{a_1}(b_1),\alpha(a_3)(b_2),\lambda_{a_3}(b_3))$.
\end{lemma}

\begin{proof}
It follows by a standard calculation, using the fact that $B_{m+1}\times...\times B_r $ is a trivial left brace.
\end{proof}

\begin{lemma}\label{automorsemi}
A permutation $\psi$ of $Sym(\bar{A}\times \bar{B})$ is an automorphism of the left brace $\bar{A}\times \bar{B}$ if and only if it is of the form $(\psi_1,\psi_2,\psi_3)$, where $\psi_1$ acts on the $j^{th}$-component of an elements belonging to $A_1\times...\times A_v$ as the ring multiplication by an invertible element of the form $1+s_j$ with $s_j\in Soc(A_j)$, for every $j\in \{1,...,v\}$, $\psi_2\in Aut(B_{m+1}\times...\times B_r,+)$ hence acts on the $j^{th}$-component of an elements belonging to $B_{m+1}\times...\times B_r$ as the ring multiplication by an invertible element of $B_{j+m}$, for every $j\in \{1,...,r-m\}$, and $\psi_3$ acts on the $j^{th}$-component of an elements belonging to $B_1\times...\times B_m$ as the ring multiplication by an invertible element of the form $1+s_j$ with $s_j\in I_j$, for every $j\in \{1,...,m\}$.
\end{lemma}

\begin{proof}
It follows by Propositions $1$ and $14$ of \cite{rump2019classification}. 
\end{proof}

Recall that an element of the left brace $A=\bar{A}\times \bar{B}$ can be uniquely written as a triple $(a,b,c)$ where $a\in \bar{A}$, $b\in B_{m+1}\times...\times B_r$ and $c\in B_1\times...\times B_m$. Moreover, $a$ can be written as $(a_1,...,a_v)\in A_1\times...\times A_v$, $b$ can be written as $(b_{m+1},...,b_r)\in B_{m+1}\times...\times B_r$ and $c$ can be written as $(c_1,...,c_m)\in B_1\times...\times B_m$. Using this notation, we give a further lemma.

\begin{lemma}\label{isoclasses}
Let $(a,b,c),(a',b',c')$ be elements of $\bar{A}\times \bar{B}$ belonging to some transitive cycle bases and let $X_{1}$ (resp. $X_2$) be the uniconnected cycle set associated to $(a,b,c) $ (resp. $(a',b',c')$). Then $X_1$ and $X_2$ are isomorphic if and only if the following conditions hold:
\begin{itemize}
    \item[1)] $a_j'=(1+s_j)(1+q^{d_j}a_j'')a_j$ for some $s_j\in Soc(A_j)$ and $a_j''\in A_j$, for all $j\in \{1,...,v\}$;
    \item[2)] $c_j'=(1+s_j)(1+p^{f_j}a_j'')c_j$ for some $s_j\in I_j$ and $a_j''\in B_j$, for all $j\in \{1,...,m\}$.
\end{itemize}
\end{lemma}

\begin{proof}
At first, note that by \cref{gentrans} $a$ and $a'$ are generators of $(\bar{A},+)$, $b$ and $b'$ are generators of $(B_{m+1}\times...\times B_r,+)$ and $c$ and $c'$ are generators of $(B_1\times...\times B_m,+)$. By \cref{isom}, $X_1$ and $X_2$ are isomorphic if and only if there exist $\psi\in Aut(A)$ and $e\in A$ such that $\psi\lambda_e(a,b,c)=(a',b',c') $.
Now, since by \cite[Proposition 5]{rump2019classification} $(B_{m+1}\times...\times B_r,+,\circ) $ is a trivial left brace, $b$ and $b'$ are in the same orbit respect to $Aut(A)$. Hence, by the Lemmas \ref{lambdamap} and \ref{automorsemi}, $X_1$ and $X_2$ are isomorphic if and only if conditions $1)$ and $2)$ of the statement hold.
\end{proof}

%\begin{lemma}
%Let $\mathcal{P}$ be the subgroup of $Sym(\bar{A}\times \bar{B})$ group generated by the set  
%$$\{\psi \lambda_{(a_1,a_2,a_3)}\ | \ (a_1,a_2,a_3)\in \bar{A}\times %\bar{B},\ \psi\in Aut(\bar{A}\times \bar{B},+,\circ) \}.$$
%\end{lemma}

Using the previous lemma we can classify all the uniconnected cycle sets with odd size, a Z-group permutation group and having a fixed associated left brace.

\begin{theor}\label{classfinal}
Let $(a,b,c),(a',b',c')$ be elements of $\bar{A}\times \bar{B}$ belonging to some transitive cycle bases and let $X_{1}$ (resp. $X_2$) be the uniconnected cycle set associated to $(a,b,c) $ (resp. $(a',b',c')$). Then $X_1$ and $X_2$ are isomorphic if and only if the following conditions hold:
\begin{itemize}
    \item[c1)] $a_j'\equiv a_j\ (mod\ q^{z_1})$ where $z_1=min\{\gamma_j-d_j,d_j\}$, for all $j\in \{1,...,v\}$ (hence $a_j$ and $a_j'$ are any if $d_j=\gamma_j$);
    \item[c2)] $c_j'\equiv c_j\ (mod\ p^{z_2})$ where $z_2=min\{\beta_j-f'_j,f_j\}$, for all $j\in \{1,...,m\}$.
\end{itemize}
\end{theor}

\begin{proof}
If c1) and c2) hold, then $a_j'=a_j+r_j q^{z_1}=a_j(1+a_j^{-1}q^{z_1}r_j)$ for some $r_j\in A_j$ and $c_j'=c_j+p^{z_2}s_j=c_j(1+c_j^{-1}p^{z_2}s_j)$ for some $s_j\in B_j$ and this fact implies 1) and 2) of the previous proposition, hence $X_1$ and $X_2$ are isomorphic.\\
Conversely, if $X_1$ and $X_2$ are isomorphic by the previous proposition we have that $a_j'=a_j+q^{d_j}a_j''a_j+s_j a_j+s_j q^{d_j}a_j'' a_j$ for some $s_j\in Soc(A_j)$ and $a_j''\in A_j$, for all $j\in \{1,...,v\}$ and $c_j'=c_j+p^{f_j}a_j''c_j+s_j c_j+s_j p^{f_j}a_j'' c_j$ for some $s_j\in I_j$ and $a_j''\in B_j$, for all $j\in \{1,...,m\}$. Using the last equalities conditions $c1)$ and $c2)$ follow, hence the thesis.
\end{proof}

\begin{ex}
Suppose that $\bar{A}=\{0\}$, $r=2$, $m=1$, $\bar{B}=p_1^{\alpha_1}p_2^{\alpha_2}$ with $p_1<p_2$ and $\alpha$ inijective. Then, the left brace $\bar{B}$ is the semidirect product $B_2\rtimes_{\alpha} B_1$  of two left braces $ B_1$ having size $p_1^{\alpha_1}$ and $B_2$ having size $p_2^{\alpha_2}$. Let $(b_2,b_1),(b_2',b_1')$ be elements of $ B_2\rtimes_{\alpha} B_1$ belonging to some transitive cycle bases and let $X_{1}$ (resp. $X_2$) be the uniconnected cycle set associated to $(b_2,b_1) $ (resp. $(b_2',b_1')$). By \cref{automorsemi}, we have that $Aut(B_2\rtimes_{\alpha} B_1,+,\circ)$ is equal to the set $\{(\psi_2,id_{B_1})\ |\ \psi_2\in Aut(B_2,+)\ \}$, therefore by \cref{isoclasses} $X_1$ and $X_2$ are isomorphic if and only if $b_1$ and $b_1'$ are in the same orbit respect to the lambda-maps of $(B_1,+,\circ)$. 
\end{ex}

\noindent \cref{classfinal} is also useful to count the non-isomorphic uniconnected cycle sets with odd size and a Z-group permutation group having a fixed left brace $\bar{A}\times \bar{B}$ as associated left brace.

\begin{prop}
The number of non-isomorphic uniconnected cycle sets with odd size and a Z-group permutation group having a fixed left brace $\bar{A}\times \bar{B}$ as associated left brace is exactly $(\prod_{i=1}^{v} k_i) \cdotp (\prod_{j=1}^{m} l_j)$ where $k_i=max\{q_i^{min\{\gamma_i-d_i,d_i\}-1}(q_i-1),1\}$ for all $i\in \{1,...,v\}$ and $l_j=p_j^{min\{\beta_j-f'_j,f_j\}-1}(p_j-1)$ for all $j\in \{1,...,m\}$
\end{prop}

\begin{proof}
By \cref{classfinal}, to show the thesis we have to count the number of the generators of $(A_i,+)$ module $q^{min\{\gamma_i-d_i,d_i\}}$, for every $i\in \{1,...,v\}$ and the number of the generators of $(B_i,+)$ module $p^{min\{\beta_j-f'_j,f_j\}}$, for every $j\in \{1,...,m\}$. The first number is equal to $max\{q_i^{min\{\gamma_i-d_i,d_i\}-1}(q_i-1),1\}$, for every $i\in \{1,...,v\}$, the second one is $p_j^{min\{\beta_j-f'_j,f_j\}-1}(p_j-1)$ for all $j\in \{1,...,m\}$, hence the thesis follows.
\end{proof}

\noindent In this section we showed that \emph{all} the uniconnected cycle set with odd size and having a Z-group permutation group can be easily constructed and classified as follows. We skip the proof, since it is a direct consequence of the previous results.

\begin{theor}\label{skipp}
Let $H=C_1\rtimes_{\alpha} C_2$ be a Z-group having odd size (by \cite[Theorem 9.4.3]{hall2018theory}, $H$ must be isomorphic to the semidirect product of cyclic groups $C_1$ and $C_2$). Then, an uniconnected cycle set with permutation group isomorphic to $H$ can be constructed as follows:
\begin{itemize}
    \item Given $\{B_{m+1},...,B_r \}$ the Sylow subgroups of $C_1$ and $\{B_{1},...,B_m \}$ the ones of $C_2$, consider the Sylow Basis $\{B_1,...,B_r \}$ of $H$;
    \item Construct all the cyclic left braces on the $B_i$ using \cite[Theorem 1]{rump2007classification};
    \item Construct $\bar{A}$ and $\bar{B}$ as in Propositions \ref{costrtutt} and \ref{cosZgroup};
    \item Give the left brace $A:=\bar{A}\times \bar{B}$, compute its multipermutation level using \cref{multilevel}
    \item Given the left brace $A:=\bar{A}\times \bar{B}$, use \cref{forind} of \cref{rump20} (where $g$ is simply a generator of the additive group of $A$ - see \cref{gentrans}) to construct all the uniconnected cycle sets $X$ having $A$ as associated left brace. By \cref{retrcyc}, $X$ has multipermutation level equal to $mpl(A)$. Finally, cut the isomorphism classes using \cref{classfinal}.
\end{itemize}
\end{theor}

We remark that this machinery works because we only have to consider cyclic left braces. Since there exist left braces with a Z-group as multiplicative group, even size and non-cyclic additive group, we can not remove the "odd size" restriction.\\
Now, we close the section showing that the cycle sets considered in this section can uniquely be determined by $5$ parameters.

\begin{lemma}\label{inv2}
A complete set of invariants for a finite Z-group $G$ is a triple of natural numbers $(m_1,n_1,r_1)$ such that $gcd( (r_1-1) n_1, m_1)=1$, $r_1\in \{0,...,m_1-1\}$ and ${r_{1}}^{n_1}\equiv 1\ (mod \ m_1)$. In this way, $G$ can be identified with the semidirect product
$(\mathbb{Z}/ m_1 \mathbb{Z},+)\rtimes_{\alpha_{r_1}}(\mathbb{Z}/n_1\mathbb{Z},+)$ where $\alpha$ is given by $\alpha_{r_1}(1)(1):=r_1$.
\end{lemma}

\begin{proof}
It follows by \cite[Chapter 9]{hall2018theory}.
\end{proof}

\begin{lemma}\label{inv4}
A complete set of invariants for a cyclic left brace $A$ having odd size and with a Z-group multiplicative group is a quadruple of natural numbers $(m_1,n_1,r_1,t)$ such that $gcd((r_1-1)n_1,m_1)=1$, $r_1\in \{0,...,m_1-1\}$ and $r_1^{n_1}\equiv 1\ (mod\ m_1)$ and $t$ is such that every prime which divides $m_1 n_1$ divides $t$. 
\end{lemma}

\begin{proof}
It follows by \cref{inv2}, \cite[Theorem 1]{rump2007classification} and \cite[Sections 4-5]{rump2019classification}.
\end{proof}

\begin{theor}\label{inv}
A complete set of invariants for a finite uniconnected cycle set $X$ having odd size and with a Z-group permutation group is the following:
\begin{itemize}
    \item[1)] a quadruple of natural numbers $(m_1,n_1,r_1,t)$ as in \cref{inv4} which determines the structure of the cyclic left brace $A=\bar{A}\times \bar{B}$ associated to $X$;
    \item[2)] a generator of the additive group $(\bar{A},+)\times (B_1,+)\times...\times (B_m,+)$, where $m$ is the natural number as in \cref{conve}.
\end{itemize}
\end{theor}

\begin{proof}
It follows by \cref{inv4} and \cref{classfinal}.
\end{proof}

Note that if a finite uniconnected cycle set $X$ has odd size and cyclic permutation group, by the previous theorem we obtain that it is uniquely determine by the quadruple $(1,|X|,1,t)$ and a generator of the additive group $(\bar{A},+)$, hence the set of invariants in this case has $3$ elements. This is consistent with \cite[Theorem  3.15]{jedlivcka2021cocyclic}.

\section{The square-free case}

In this section, we apply the results obtained before to the uniconnected cycle sets of square-free odd order. To exhibit the result in a comfortable way, in the following we use the same notation of the previous section.

\vspace{2mm}

\noindent At first, we show that the uniconnected cycle sets of square-free odd order belong to the family of the ones having a Z-group permutation group and that they are all of multipermutation level at most $2$. 

\begin{prop}\label{sqf2}
Let $X$ be an uniconnected cycle set of square-free odd order. Then, $\mathcal{G}(X)$ is a Z-group and $X$ has multipermutation level at most $2$. Moreover, it has multipermutation level $1$ if and only if the permutation group $\mathcal{G}(X)$ is abelian.
\end{prop}

\begin{proof}
By \cref{rump20}, there exist a left brace $A$ such that $X$ is isomorphic to an uniconnected associated cycle set $(A,\bullet)$. Since $(A,\circ)$ (which is equal to $\mathcal{G}(X)$) has square-free order, by \cite[Chapter 9]{hall2018theory} it is a $Z-$group and by  \cite[Corollary 2]{rump2019classification} $A$ is a cyclic left brace. By \cite[Chapter 9]{hall2018theory} and \cite[Sections 3-4]{rump2019classification}, $A$ must be isomorphic to a semidirect product of two trivial cyclic left braces $C_1\rtimes C_2$. By \cref{multilevel}, $A$ has multipermutation level at most $2$ and by \cref{retrcyc} the first part of the thesis follows.\\
Now, if $X$ has multipermutation level $1$ then $\mathcal{G}(X)$ is cyclic and hence abelian. Conversely, if  $\mathcal{G}(X)$ is abelian it follows that $(A,\circ)$ is abelian and hence $A$ is a trivial cyclic left brace of multipermutation level $1$, therefore by \cref{retrcyc} $X$ also has multipermutation level $1$.
\end{proof}

%The results of the previous section imply that a cyclic left brace of square-free odd order and with non-abelian multiplicative group can be constructed by a direct product of cyclic left braces $\bar{A}\times \bar{B}$, where $\bar{A}$ is a trivial cyclic left brace having square-free order and $\bar{B}$ is an other cyclic left brace of square-free order, having size coprime with the one of $\bar{A}$, obtained as a semidirect product of trivial cyclic left braces $B_1$ and $B_2$ such that every Sylow subgroup of $B_1$ acts non-trivially to at least a Sylow subgroup of $B_2$ and the whole action of $B_1$ on every Sylow subgroup of $B_2$ is not trivial. This fact is of crucial importance to distinguish the isomorphism classes of uniconnected cycle sets of odd square-free order.

\begin{cor}\label{teosq}
Let $A:=\bar{A}\times \bar{B}$ be a cyclic left brace of square-free odd order. Moreover, let $g:=(g_1,g_2,g_3)\in \bar{A}\times \bar{B}$ be an element of a transitive cycle base of $A$, where $g_1\in \bar{A}$, $g_2\in B_{m+1}\times...\times B_r$ and $g_3\in B_1\times...\times B_m$. Then, the binary operation $\bullet_g$ on $A$ given by
$$(a_1,a_2,a_3)\bullet_g (b_1,b_2,b_3):=\lambda_{(a_1,a_2,a_3)}(g_1,g_2,g_3)^- \circ (b_1,b_2,b_3) $$
for all $(a_1,a_2,a_3), (b_1,b_2,b_3)\in A $ makes $A$ into an uniconnected cycle set having multipermutation level $2$. Moreover, if $h:=(h_1,h_2,h_3)$ is another element of a transitive cycle base of $A$ and $\bullet_h$ the binary operation on $A$ defined as $\bullet_g$, we have that $(A,\bullet_g)$ is isomorphic to $(A,\bullet_{h})$ if and only if $h_3=g_3$.
\end{cor}

\begin{proof}
Since in this case $d_j=\gamma_j$ for all $j\in \{1,...,v\}$, $f_j=\beta_j$ and $f'_j=0$ for all $j\in \{1,...,m\}$, the result follows by \cref{classfinal}. 
\end{proof}

\noindent It is well-known that all the group having square-free order are Z-group. Moreover, by a result due by Rump, given a group $H$ of square-free order, there exist a unique left brace $A$ such that $(A,\circ)\cong H$ (see \cite[Corollary 2]{rump2019classification}) and it always is a semidirect product of trivial left braces. Therefore, the previous theorem allow to give a precise classification of all the uniconnected cycle sets of square-free odd order.

\begin{cor}\label{classifsq}
Let $X$ be an uniconnected cycle set of square-free odd order with multipermutation level $2$ and let $A$ be a cyclic left brace such that $X$ is an uniconnected associated cycle set. Moreover, with the same notation of the previous corollary, let $g:=(g_1,g_2,g_3)$ be an element of $A$ such that $X\cong (A,\bullet_g)$. Let $Y$ be an other uniconnected cycle set such that $\mathcal{G}(Y)\cong \mathcal{G}(X)$. Then, there exist an element $h:=(h_1,h_2,h_3)\in A$ such that $Y\cong (A,\bullet_h)$ and $X\cong Y$ if and only if $g_3=h_3$.
\end{cor}

\begin{proof}
Since $\mathcal{G}(Y)\cong \mathcal{G}(X) $, by \cite[Corollary 2]{rump2019classification} it follows that $A$ is the cyclic left brace associated to $Y$ hence by \cref{rump20} there exist an element $h:=(h_1,h_2,h_3)\in A$ such that $Y\cong (A,\bullet_h)$. The rest of the proof follows by \cref{teosq}.
\end{proof}

\begin{cor}
Let $X$ be an uniconnected cycle set having square-free order and with permutation group $H$. Then, $X$ has multipermutation level $1$ or it has multipermutation level $2$ and it is an uniconnected associated cycle set constructed as in \cref{teosq} starting from the unique left brace $A$ having multiplicative group isomorphic to $H$.
\end{cor}

\begin{proof}
Straighforward.
\end{proof}

\noindent By \cref{sqf2}, for every square-free number $n$ there is only one uniconnected cycle set having order $n$ and with abelian permutation group: it is the one having multipermutation level $1$. Now, we give an explicit formula to count the non-isomorphic uniconnected cycle sets of square-free odd order with non-abelian permutation group. 

\begin{cor}
Let $H$ be a non-abelian group having square-free odd order and let $(A,+,\circ)$ be the cyclic left brace such that $(A,\circ)\cong H$. Then, the number of non-isomorphic uniconnected cycle sets of square-free order having permutation group isomorphic to $H$ is $\phi(|B_1\times...\times B_m|)$, where $\phi$ is the Euler function.
\end{cor}

\begin{proof}
By \cref{gentrans} and \cref{classifsq}, it is sufficient to count the possible values $g_3$ of a generator $(g_1,g_2,g_3)$ of the additive group $(A,+)$, and these clearly are $\phi(|B_1\times...\times B_m|)$.
\end{proof}

\bibliographystyle{elsart-num-sort}
\bibliography{Bibliography}

\def\cprime{$'$}
\begin{thebibliography}{10}
\expandafter\ifx\csname url\endcsname\relax
  \def\url#1{\texttt{#1}}\fi
\expandafter\ifx\csname urlprefix\endcsname\relax\def\urlprefix{URL }\fi

\bibitem{bachiller2015classification}
D.~Bachiller, Classification of braces of order p3, Journal of Pure and Applied
  Algebra 219~(8) (2015) 3568--3603.
\newline\urlprefix\url{https://doi.org/10.1515/forum-2015-0240}

\bibitem{bachiller2016solutions}
D.~Bachiller, F.~Ced{\'o}, E.~Jespers, Solutions of the {Yang--Baxter} equation
  associated with a left brace, J. Algebra 463 (2016) 80--102.

\bibitem{cacsp2018}
M.~Castelli, F.~Catino, G.~Pinto, Indecomposable involutive set-theoretic
  solutions of the {Yang-Baxter} equation, J. Pure Appl. Algebra 220~(10)
  (2019) 4477--4493.
\newline\urlprefix\url{https://doi.org/10.1016/j.jpaa.2019.01.017}

\bibitem{capiru2020}
M.~Castelli, G.~Pinto, W.~Rump, On the indecomposable involutive set-theoretic
  solutions of the {Y}ang-{B}axter equation of prime-power size, Comm. Algebra
  48~(5) (2020) 1941--1955.
\newline\urlprefix\url{https://doi.org/10.1080/00927872.2019.1710163}

\bibitem{cedo2014braces}
F.~Ced{\'o}, E.~Jespers, J.~Okni{\'n}ski, Braces and the {Yang-Baxter}
  equation, Comm. Math. Phys. 327~(1) (2014) 101--116.
\newline\urlprefix\url{https://doi.org/10.1007/s00220-014-1935-y}

\bibitem{cedo2020primitive}
F.~Ced{\'o}, E.~Jespers, J.~Okninski, Primitive set-theoretic solutions of the
  {Yang-Baxter} equation, arXiv preprint arXiv:2003.01983.
\newline\urlprefix\url{https://arxiv.org/pdf/2003.01983.pdf}

\bibitem{drinfeld1992some}
V.~G. Drinfel\cprime~d, On some unsolved problems in quantum group theory, in:
  Quantum groups ({L}eningrad, 1990), vol. 1510 of Lecture Notes in Math.,
  Springer, Berlin, 1992, pp. 1--8.
\newline\urlprefix\url{https://doi.org/10.1007/BFb0101175}

\bibitem{etingof1998set}
P.~Etingof, T.~Schedler, A.~Soloviev, Set-theoretical solutions to the {Quantum
  Yang-Baxter} equation, Duke Math. J. 100~(2) (1999) 169--209.
\newline\urlprefix\url{http://doi.org/10.1215/S0012-7094-99-10007-X}

\bibitem{gateva2008matched}
T.~Gateva-Ivanova, S.~Majid, Matched pairs approach to set theoretic solutions
  of the {Yang--Baxter} equation, J. Algebra 319~(4) (2008) 1462--1529.
\newline\urlprefix\url{https://doi.org/10.1016/j.jalgebra.2007.10.035}

\bibitem{gateva1998semigroups}
T.~Gateva-Ivanova, M.~Van~den Bergh, Semigroups of {I-Type}, J. Algebra 206~(1)
  (1998) 97--112.
\newline\urlprefix\url{https://doi.org/10.1006/jabr.1997.7399}

\bibitem{hall2018theory}
M.~Hall, The theory of groups, Courier Dover Publications, 2018.

\bibitem{JePiZa20x}
P.~Jedli{\v{c}}ka, A.~Pilitowska, A.~Zamojska-Dzienio, Indecomposable
  involutive solutions of the {Y}ang-{B}axter equation of multipermutational
  level $2$ with abelian permutation group, Forum Math. 2020.
\newline\urlprefix\url{https://doi.org/10.1515/forum-2021-0130}

\bibitem{jedlivcka2021cocyclic}
P.~Jedli{\v{c}}ka, A.~Pilitowska, A.~Zamojska-Dzienio, Cocyclic braces and
  indecomposable cocyclic solutions of the {Yang-Baxter} equation, arXiv
  preprint arXiv:2107.12319.
\newline\urlprefix\url{https://arxiv.org/pdf/2107.12319.pdf}

\bibitem{RaVe21}
S.~Ram{\'i}rez, L.~Vendramin, Decomposition theorems for involutive solutions
  to the {Y}ang-{B}axter equation, Int. Math. Res. Not. IMRN.
\newline\urlprefix\url{https://doi.org/10.1093/imrn/rnab232}

\bibitem{robinson2012course}
D.~J. Robinson, A Course in the Theory of Groups, vol.~80, Springer Science \&
  Business Media, 2012.
\newline\urlprefix\url{http://doi.org/10.1007/978-1-4419-8594-1}

\bibitem{rump2005decomposition}
W.~Rump, A decomposition theorem for square-free unitary solutions of the
  quantum {Y}ang-{B}axter equation, Adv. Math. 193 (2005) 40--55.
\newline\urlprefix\url{https://doi.org/10.1016/j.aim.2004.03.019}

\bibitem{rump2007braces}
W.~Rump, Braces, radical rings, and the quantum {Y}ang-{B}axter equation, J.
  Algebra 307~(1) (2007) 153--170.
\newline\urlprefix\url{https://doi.org/10.1016/j.jalgebra.2006.03.040}

\bibitem{rump2007classification}
W.~Rump, Classification of cyclic braces, Journal Pure Appl. Algebra 209~(3)
  (2007) 671--685.
\newline\urlprefix\url{https://doi.org/10.1016/j.jpaa.2006.07.001}

\bibitem{rump2019classification}
W.~Rump, Classification of cyclic braces, {II}, Trans. Amer. Math. 372~(1)
  (2019) 305--328.
\newline\urlprefix\url{https://doi.org/10.1090/TRAN%2F7569}

\bibitem{rump2020}
W.~Rump, Classification of indecomposable involutive set-theoretic solutions to
  the {Y}ang-{B}axter equation, Forum Math. 32~(4) (2020) 891--903.
\newline\urlprefix\url{https://doi.org/10.1515/forum-2019-0274}

\bibitem{Ru20}
W.~Rump, Cocyclic solutions to the {Y}ang-{B}axter equation, Proc. Amer. Math.
  Soc. 149~(2) (2021) 471--479.
\newline\urlprefix\url{https://doi.org/10.1090/proc/15220}

\bibitem{smock}
A.~Smoktunowicz, A.~Smoktunowicz, Set-theoretic solutions of the
  {Y}ang-{B}axter equation and new classes of {R}-matrices, Linear Algebra
  Appl. 546 (2018) 86--114.
\newline\urlprefix\url{https://doi.org/10.1016/j.laa.2018.02.001}

\end{thebibliography}

\end{document}